\documentclass[12pt]{article}

\usepackage{amsmath, amssymb}
\usepackage{amsthm}

\usepackage{color}
\definecolor{darkred}{RGB}{139,0,0}
\definecolor{darkgreen}{RGB}{0,100,0}
\definecolor{darkmagenta}{RGB}{139,0,139}
\definecolor{darkpurple}{RGB}{110,0,180}
\definecolor{darkblue}{RGB}{40,0,200}
\definecolor{darkorange}{RGB}{255,140,0}

\newcommand{\bsx}{\boldsymbol{x}}

\newcommand{\bszero}{\boldsymbol{0}}

\newcommand{\bsdelta}{\boldsymbol{\delta}}

\newcommand{\bsk}{\boldsymbol{k}}

\newcommand{\bsy}{\boldsymbol{y}}

\newcommand{\bsw}{\boldsymbol{w}}
\newcommand{\rd}{\,{\rm d}}

\newcommand{\NN}{\mathbb{N}}
\newcommand{\ZZ}{\mathbb{Z}}
\newcommand{\EE}{\mathbb{E}}
\newcommand{\cP}{\mathcal{P}}

\newcommand{\cQ}{\mathcal{Q}}
\newcommand{\cR}{\mathcal{R}}
\newcommand{\icomp}{\mathtt{i}}

\newtheorem{theorem}{Theorem}
\newtheorem{proposition}{Proposition}
\newtheorem{remark}{Remark}
\newtheorem{lemma}{Lemma}

\begin{document}

\title{A note on the periodic $L_2$-discrepancy of Korobov's $p$-sets}

\author{Josef Dick\thanks{This project started while Josef Dick was ``Land Ober\"osterreich Guest Professor'' within the Special Research Program ``Quasi-Monte Carlo Methods: Theory and Applications'' which is funded by the Austrian Science Fund (FWF) Project F55-N26, and the gouverment of the Austrian state Upper Austria. It was completed while he visited JKU during a further research stay. },
Aicke Hinrichs\thanks{A. Hinrichs is supported by the Austrian Science Fund (FWF) Project F5513-N26, which is a part of the Special Research Program ``Quasi-Monte Carlo Methods: Theory and Applications''.
}\,
and 
Friedrich Pillichshammer\thanks{F. Pillichshammer is supported by the Austrian Science Fund (FWF) Project F5509-N26, which is a part of the Special Research Program ``Quasi-Monte Carlo Methods: Theory and Applications''.}
}

\date{}

\maketitle

\begin{abstract}
We study the periodic $L_2$-discrepancy of point sets in the $d$-dimensional torus. This discrepancy is intimately connected with the root-mean-square $L_2$-discrepancy of shifted point sets, with the notion of diaphony, and with the worst case error of cubature formulas for the integration of periodic functions in Sobolev spaces of mixed smoothness. 

In discrepancy theory many results are based on averaging arguments. In order to make such results relevant for applications one requires explicit constructions of point sets with ``average'' discrepancy. In our main result we study Korobov's $p$-sets and show that this point sets have periodic $L_2$-discrepancy of average order. This result is related to an open question of Novak and Wo\'{z}niakowski. 
\end{abstract}

\section{Introduction}

In this note we study the periodic $L_2$-discrepancy which is a quantitative measure for the irregularity of distribution of a point set, but which is also closely related to the worst-case integration error of quasi-Monte Carlo integration rules (see, for example, \cite{hin,HOe,HW,Lev}). In order to state its definition we first explain the ``test-sets'' that are considered in this specific notion of discrepancy.

For $x,y \in [0,1)$ we define the periodic ``interval'' $I(x,y)$ as $$I(x,y)=\left\{ 
\begin{array}{ll}
[x,y) & \mbox{if $x \le y$},\\
\left[0,y\right) \cup [x,1) &  \mbox{if $x > y$.}
\end{array}  \right.
$$
For dimension $d >1$ and $\bsx=(x_1,\ldots,x_d)$ and $\bsy=(y_1,\ldots,y_d)$ in $[0,1)^d$ the periodic ``boxes'' $B(\bsx,\bsy)$ are given by $$B(\bsx,\bsy)=I(x_1,y_1)\times \ldots \times I(x_d,y_d).$$

The local discrepancy of a point set $\cP=\{\bsx_1,\bsx_2,\ldots,\bsx_{N}\}$ consisting of $N$ elements in the $d$-dimensional unit cube with respect to a periodic box $B=B(\bsx,\bsy)$ is given by $$\Delta_{\cP}(B)=\frac{\#\{j \in \{1,\ldots,N\}\ : \ \bsx_j \in B\}}{N}-{\rm volume}(B).$$ 
Then the {\it periodic $L_2$-discrepancy} of $\cP$ is the $L_2$-norm of the local discrepancy taken over all periodic boxes $B=B(\bsx,\bsy)$, i.e., $$L_{2,N}^{{\rm per}}(\cP)=\left(\int_{[0,1]^d} \int_{[0,1]^d} \Delta_{\cP}(B(\bsx,\bsy))^2 \rd \bsx \rd \bsy\right)^{1/2}.$$

The usual $L_2$-discrepancy of a point set $\cP$ is defined as $$L_{2,N}(\cP)=\left(\int_{[0,1]^d} \Delta_{\cP}(B(\bszero,\bsy))^2 \rd \bsy\right)^{1/2}.$$ 

In the closely connected context of worst case errors for cubature formulas it is natural to extend these notions to weighted point sets. If, additionally to the point set $\cP=\{\bsx_1,\bsx_2,\ldots,\bsx_{N}\}$ we also have a set of associated real weights  $\bsw=\{w_1,w_2,\ldots,w_{N}\}$, the local discrepancy of the weighted point set $\cP$ is given as 
$$\Delta_{\cP}(B,\bsw)=
\left(\sum_{j: x_j\in B} w_j\right) -{\rm volume}(B).$$
Then $\Delta_{\cP}(B) = \Delta_{\cP}(B,\bsw)$ is obtained for equal weights $w_j=1/N$.
The {\it periodic $L_2$-discrepancy} of the weighted point set $\cP$ is 
$$L_{2,N}^{{\rm per}}(\cP,\bsw)=\left(\int_{[0,1]^d} \int_{[0,1]^d} \Delta_{\cP}(B(\bsx,\bsy),\bsw)^2 \rd \bsx \rd \bsy\right)^{1/2}.$$ 
The usual $L_2$-discrepancy of the weighted point set $\cP$ is 
$$L_{2,N}(\cP,\bsw)=\left(\int_{[0,1]^d} \Delta_{\cP}(B(\bszero,\bsy),\bsw)^2 \rd \bsy\right)^{1/2}.$$ 

For a (weighted) point set $\cP=\{\bsx_1,\bsx_2,\ldots,\bsx_{N}\}$ and a real vector $\bsdelta \in [0,1)^d$ the shifted point set $\cP+\bsdelta$ is defined as $$\cP+\bsdelta=\{\{\bsx_1+\bsdelta\},\{\bsx_2+\bsdelta\},\ldots,\{\bsx_{N}+\bsdelta\}\},$$ where $\{\bsx_j+\bsdelta\}$ means that the fractional-part-function $\{x\}=x-\lfloor x \rfloor$ for non-negative real numbers $x$, is applied component-wise to the vector $\bsx_j+\bsdelta$. The root-mean-square $L_2$-discrepancy of a shifted (and weighted) point set $\cP$ with respect to all uniformly distributed shift vectors $\bsdelta \in [0,1)^d$ is 
$$\sqrt{\EE_{\bsdelta}[(L_{2,N}(\cP+\bsdelta,\bsw))^2]}=\left(\int_{[0,1]^d} (L_{2,N}(\cP+\bsdelta,\bsw))^2 \rd \bsdelta \right)^{1/2}.$$ 
Note that $$\EE_{\bsdelta}[L_{2,N}(\cP+\bsdelta,\bsw)] \le \sqrt{\EE_{\bsdelta}[(L_{2,N}(\cP+\bsdelta,\bsw))^2]}.$$
If the weights are the standard equal weights $w_j=1/N$, we drop $\bsw$ from the notation.

The following relation between periodic $L_2$-discrepancy and root-mean-square $L_2$-discrepancy of a shifted point set $\cP$ holds (see also \cite{Lev}):

\begin{proposition}\label{pr1}
For $\cP=\{\bsx_1,\bsx_2,\ldots,\bsx_{N}\}$ in $[0,1)^d$ and weights $\bsw=\{w_1,w_2,\ldots,w_{N}\}$ we have $$L_{2,N}^{{\rm per}}(\cP,\bsw)= \sqrt{\EE_{\bsdelta}[(L_{2,N}(\cP+\bsdelta,\bsw))^2]}.$$ In particular, we have $$L_{2,N}^{{\rm per}}(\cP)= \sqrt{\EE_{\bsdelta}[(L_{2,N}(\cP+\bsdelta))^2]}.$$
\end{proposition}

\begin{proof}
We write down the proof for dimension $d=1$ and the equal weighted case only. The general case follows by similar arguments. We have
\begin{eqnarray*}
\EE_{\delta}[(L_{2,N}(\cP+\delta))^2] & = & \int_0^1 \int_0^1 \Delta_{\cP+\delta}([0,y))^2 \rd y \rd \delta \\
& = &  \int_0^1 \left( \int_0^y \Delta_{\cP+\delta}([0,y))^2 \rd \delta + \int_y^1 \Delta_{\cP+\delta}([0,y))^2 \rd \delta\right)\rd y.
\end{eqnarray*}
We consider two cases:
\begin{itemize}
\item If $\delta < y$, then $\{x_j+\delta\} \in [0,y)$ iff $x_j+\delta \in [0,y)$ or $x_j+\delta-1 \in [0,y)$ and this holds iff $x_j \in [0,y-\delta)\cup [1-\delta,1)$. Hence $$\{x_j+\delta\} \in [0,y) \ \Leftrightarrow \ x_j \in I(1-\delta,y-\delta).$$ Note that ${\rm volume}(I(1-\delta,y-\delta))=y$.
\item If $\delta\ge y$, then $\{x_j+\delta\} \in [0,y)$ iff $x_j+\delta \in [0,y)$ or $x_j+\delta-1 \in [0,y)$ and this holds iff $x_j \in [1-\delta,1+y-\delta)$. Hence $$\{x_j+\delta\} \in [0,y) \ \Leftrightarrow \ x_j \in I(1-\delta,1+y-\delta).$$  Note that ${\rm volume}( I(1-\delta,1+y-\delta))=y$.
\end{itemize}
Hence,
\begin{eqnarray*}
\lefteqn{\EE_{\delta}[(L_{2,N}(\cP+\delta))^2]}\\
& = & \int_0^1 \left( \int_0^y \Delta_{\cP}(I(1-\delta,y-\delta))^2 \rd \delta + \int_y^1 \Delta_{\cP}(I(1-\delta,1+y-\delta))^2 \rd \delta\right)\rd y\\
& = & \int_0^1 \left( \int_{1-y}^1 \Delta_{\cP}(I(t,y+t-1))^2 \rd t + \int_0^{1-y} \Delta_{\cP}(I(t,y+t))^2 \rd t\right)\rd y\\
& = & \int_0^1\int_{1-t}^1 \Delta_{\cP}(I(t,y+t-1))^2 \rd y \rd t + \int_0^1 \int_0^{1-t}  \Delta_{\cP}(I(t,y+t))^2 \rd y \rd t\\
& = & \int_0^1\int_0^t \Delta_{\cP}(I(t,z))^2 \rd z \rd t + \int_0^1 \int_t^1  \Delta_{\cP}(I(t,z))^2 \rd z \rd t\\
& = & \int_0^1\int_0^1 \Delta_{\cP}(I(t,z))^2 \rd z \rd t \\
& = & (L_{2,N}^{{\rm per}}(\cP))^2,
\end{eqnarray*} 
where we just applied several elementary substitutions.
\end{proof}

Another important fact is, that the periodic $L_2$-discrepancy can be expressed in terms of exponential sums.

\begin{proposition}\label{pr_dia}
We have $$(L_{2,N}^{{\rm per}}(\cP))^2=\frac{1}{3^d} \sum_{\bsk \in \ZZ^d\setminus\{\bszero\}} \frac{1}{r(\bsk)^2} \left|\frac{1}{N} \sum_{h=1}^N \exp(2 \pi \icomp \bsk \cdot \bsx_h)\right|^2,$$ where $\icomp=\sqrt{-1}$ and where for $\bsk=(k_1,\ldots,k_d)\in \ZZ^d$ we set $$r(\bsk)=\prod_{j=1}^d r(k_j) \ \ \ \mbox{ and } \ \ r(k_j)=\left\{ 
\begin{array}{ll}
1 & \mbox{ if $k_j=0$},\\
\frac{2 \pi |k_j|}{\sqrt{6}} & \mbox{ if $k_j\not=0$.} 
\end{array}\right.$$
\end{proposition}

\begin{proof}
See \cite[p.~390]{HOe}. 
\end{proof}

The above formula shows that the periodic $L_2$-discrepancy is---up to a multiplicative factor---exactly the diaphony which is a well-known measure for the irregularity of distribution of point sets and which was introduced by Zinterhof~\cite{zint} in the year 1976. \\

In Section~\ref{Kps} we will estimate the periodic $L_2$-discrepancy of certain types of point sets that recently gained some attention in the context of error bounds with favourable weak dependence on the dimension $d$ (see \cite{dick2014,DP15}). On the other hand, we also show in Section~\ref{cod} that even the weighted version of the periodic $L_2$-discrepancy suffers from the curse of dimensionality.

\section{The periodic $L_2$-discrepancy of Korobov's $p$-sets}\label{Kps}

Let $p$ be a prime number. We consider the following point sets in $[0,1)^d$: 
\begin{itemize}
\item Let $\cP_p^{{\rm Kor}}=\{\bsx_0,\ldots,\bsx_{p-1}\}$ with $$\bsx_n=\left(\left\{\frac{n}{p}\right\},\left\{\frac{n^2}{p}\right\},\ldots,\left\{\frac{n^d}{p}\right\}\right)\ \ \ \mbox{ for }\ n=0,1,\ldots,p-1.$$ The point set $\cP_p^{{\rm Kor}}$ was introduced by Korobov \cite{kor} (see also \cite[Section 4.3]{HuWa}).
\item Let $\cQ_{p^2}^{{\rm Kor}}=\{\bsx_0,\ldots,\bsx_{p^2-1}\}$ with $$\bsx_n=\left(\left\{\frac{n}{p^2}\right\},\left\{\frac{n^2}{p^2}\right\},\ldots,\left\{\frac{n^d}{p^2}\right\}\right)\ \ \ \mbox{ for }\ n=0,1,\ldots,p^2-1.$$ The point set $\cQ_{p^2}^{{\rm Kor}}$ was introduced by Korobov \cite{kor1957} (see also \cite[Section 4.3]{HuWa}).
\item Let $\cR_{p^2}^{{\rm Kor}}=\{\bsx_{a,k}\ : \ a,k \in \{0,\ldots,p-1\}\}$ with $$\bsx_{a,k}=\left(\left\{\frac{k}{p}\right\},\left\{\frac{a k}{p}\right\},\ldots,\left\{\frac{a^{d-1} k}{p}\right\}\right)\ \ \ \mbox{ for }\ a,k=0,1,\ldots,p-1.$$ Note that $\cR_{p^2}^{{\rm Kor}}$ is the multi-set union of all Korobov lattice point sets with modulus $p$. The point set $\cR_{p^2}^{{\rm Kor}}$ was introduced by Hua and Wang (see \cite[Section 4.3]{HuWa}).
\end{itemize}

Hua and Wang \cite{HuWa} called the point sets  $\cP_p^{{\rm Kor}}$, $\cQ_{p^2}^{{\rm Kor}}$ and $\cR_{p^2}^{{\rm Kor}}$ the {\it $p$-sets}. We have $$| \cP_p^{{\rm Kor}}|=p \ \ \mbox{ and }\ \ |\cQ_{p^2}^{{\rm Kor}}|=|\cR_{p^2}^{{\rm Kor}}|=p^2.$$

\begin{theorem}\label{thm1}
For Korobov's $p$-sets $\cP \in \{\cP^{{\rm Kor}}_p,\cQ^{{\rm Kor}}_{p^2},\cR^{{\rm Kor}}_{p^2}\}$ we have
$$L_{2,N}^{{\rm per}}(\cP) \le \frac{d}{2^{d/2}}  \frac{1}{\sqrt{N}}.$$
\end{theorem}

For the proof of Theorem~\ref{thm1} need the following lemma.

\begin{lemma}\label{le3}
Let $p$ be a prime number and let $d \in \NN$. Then for all $h_1,\ldots,h_d\in \ZZ$ such that $p \nmid h_j$ for at least one $j \in \{1,2,\ldots,d\}$ we have 
\begin{equation}\label{WeilP}
\left|\sum_{n=0}^{p-1} \exp(2 \pi \icomp(h_1 n+h_2 n^2+\cdots+h_d n^d)/p)\right| \le (d-1) \sqrt{p},
\end{equation}

\begin{equation}\label{WeilQ}
\left|\sum_{n=0}^{p^2-1} \exp(2 \pi \icomp(h_1 n+h_2 n^2+\cdots+h_d n^d)/p^2)\right| \le (d-1) p. 
\end{equation}

\begin{equation}\label{WeilR}
\left|\sum_{a=0}^{p-1} \sum_{k=0}^{p-1} \exp(2 \pi \icomp k (h_1 +h_2 a+\cdots+h_d a^{d-1})/p)\right| \le (d-1) p.
\end{equation}

\end{lemma}

\begin{proof}
Eq.~\eqref{WeilP} follows from a bound from A. Weil \cite{weil} on exponential sums which is widely known as {\it Weil bound} (see also \cite[Theorem~5.38]{lini}). For details we refer to~\cite{dick2014}. For a proof of Eq.~\eqref{WeilQ} we refer to \cite[Lemma~4.6]{HuWa}. A proof of Eq.~\eqref{WeilR} can be found in \cite{DP15}.
\end{proof}

\begin{proof}[Proof of Theorem~\ref{thm1}]
We use the formula from Proposition~\ref{pr_dia} for the periodic $L_2$-discrepancy and estimate the  exponential sum with the help of Lemma~\ref{le3}. We provide the details only for $\cP_p^{{\rm Kor}}$ and $\cQ_{p^2}^{{\rm Kor}}$. The proof for $\cR_{p^2}^{{\rm Kor}}$ is analoguous.

Before we start we mention the following easy results that will be used later in the proof. For $\bsk=(k_1,\ldots,k_d)\in \ZZ^d$ we write $N |\bsk$ to indicate that $N| k_j$ for all $j \in \{1,2,\ldots,d\}$. Likewise, $N \nmid\bsk$ indicates that there is at least one index $j\in \{1,2,\ldots,d\}$ with $N \nmid k_j$. We have 
\begin{eqnarray*}
\sum_{\bsk\in \ZZ^d} \frac{1}{r(\bsk)^2} = \left(1+ 2 \sum_{k=1}^{\infty} \frac{6}{4 \pi^2 k^2}\right)^d =\left(1+2  \frac{6}{4 \pi^2} \frac{\pi^2}{6}\right)^d= \left(\frac{3}{2}\right)^d
\end{eqnarray*}
and 
\begin{eqnarray*}
\sum_{\bsk\in \ZZ \atop N \nmid \bsk} \frac{1}{r(\bsk)^2} & = & \sum_{\bsk \in \ZZ^d} \frac{1}{r(\bsk)^2} - \sum_{\bsk\in \ZZ\atop N | \bsk} \frac{1}{r(\bsk)^2} \\
& = & \left(\frac{3}{2}\right)^d - \left(1+ 2 \sum_{k=1}^{\infty}  \frac{6}{4 \pi^2 (N k)^2}\right)^d\\
& = & \left(\frac{3}{2}\right)^d - \left(1+ \frac{1}{2 N^2}\right)^d.
\end{eqnarray*}

First we study $\cP_p^{{\rm Kor}}$. Then $N=|\cP_p^{{\rm Kor}}|=p$. Using Eq.~\eqref{WeilP} from Lemma~\ref{le3} we have
\begin{eqnarray*}
\lefteqn{\sum_{\bsk \in \ZZ^d\setminus\{\bszero\}} \frac{1}{r(\bsk)^2} \left|\frac{1}{N} \sum_{h=1}^N \exp(2 \pi \icomp \bsk \cdot \bsx_h)\right|^2}\\
& \le &  \sum_{\bsk \in \ZZ^d \atop N \nmid \bsk} \frac{1}{r(\bsk)^2} \frac{(d-1)^2}{N}+ \sum_{\bsk \in \ZZ^d\setminus\{\bszero\}\atop N | \bsk} \frac{1}{r(\bsk)^2} \left|\frac{1}{N} \sum_{h=0}^{N-1} \exp\left(\frac{2 \pi \icomp \bsk \cdot (h,h^2,\ldots,h^d)}{N}\right)\right|^2\\
& = &  \frac{(d-1)^2}{N} \sum_{\bsk \in \ZZ^d\atop N \nmid \bsk} \frac{1}{r(\bsk)^2} + \sum_{\bsk \in \ZZ^d\setminus\{\bszero\}} \frac{1}{r(N \bsk)^2}\\
& = &  \frac{(d-1)^2}{N} \left(\left(\frac{3}{2}\right)^d- \left(1+\frac{1}{2 N^2}\right)^d\right)+\left(1+\frac{1}{2 N^2}\right)^d-1\\
& \le & \frac{(d-1)^2}{N} \left(\frac{3}{2}\right)^d+ \frac{1}{N^2} \left(\frac{3}{2}\right)^d\\
& \le & \frac{d^2}{N}  \left(\frac{3}{2}\right)^d.
\end{eqnarray*}
Inserting this bound into the formula given in Proposition~\ref{pr_dia} gives $$(L_{2,N}^{{\rm per}}(\cP^{{\rm Kor}}_p))^2\le \frac{1}{3^d}  \frac{d^2}{N}  \left(\frac{3}{2}\right)^d =  \frac{d^2}{2^d}  \frac{1}{N}.$$ This yields the desired result for $\cP_p^{{\rm Kor}}$.

Now we turn to $\cQ_{p^2}^{{\rm Kor}}$. Here $N=|\cQ_{p^2}^{{\rm Kor}}|=p^2$. Using Eq.~\eqref{WeilQ} from Lemma~\ref{le3} we have 
\begin{eqnarray*}
\lefteqn{\sum_{\bsk \in \ZZ^d\setminus\{\bszero\}} \frac{1}{r(\bsk)^2} \left|\frac{1}{N} \sum_{h=1}^N \exp(2 \pi \icomp \bsk \cdot \bsx_h)\right|^2}\\
& \le &  \sum_{\bsk \in \ZZ^d \atop p \nmid \bsk} \frac{1}{r(\bsk)^2} \frac{(d-1)^2}{N}+ \sum_{\bsk \in \ZZ^d\setminus\{\bszero\}\atop p | \bsk} \frac{1}{r(\bsk)^2} \left|\frac{1}{N} \sum_{h=0}^{N-1} \exp\left(\frac{2 \pi \icomp \bsk \cdot (h,h^2,\ldots,h^d)}{N}\right)\right|^2\\
& \le &  \sum_{\bsk \in \ZZ^d \atop p \nmid \bsk} \frac{1}{r(\bsk)^2} \frac{(d-1)^2}{N}+ \sum_{\bsk \in \ZZ^d\setminus\{\bszero\}} \frac{1}{r(p \bsk)^2}\\
& \le & \frac{(d-1)^2}{N} \left(\left(\frac{3}{2}\right)^d -\left(1+\frac{1}{2 p^2}\right)^d \right) + \frac{1}{p^2} \left( \left(\frac{3}{2}\right)^d- 1\right)\\
& \le & \frac{d^2}{N} \left(\frac{3}{2}\right)^d.
\end{eqnarray*}
Hence
$$(L_{2,N}^{{\rm per}}(\cQ^{{\rm Kor}}_{p^2}))^2\le \frac{d^2}{2^d}  \frac{1}{N} .$$ This yields the desired result for $\cQ_{p^2}^{{\rm Kor}}$.

In order to prove the bound on $L_{2,N}^{{\rm per}}(\cR_{p^2}^{{\rm Kor}})$ use Eq.~\eqref{WeilR} from Lemma~\ref{le3}.
\end{proof}

\begin{remark}\rm
Note that the dependence on the dimension of our bound on the periodic $L_2$-discrepancy is only $d 2^{-d/2}$, which looks very promising regarding tractability properties at first sight. However, it can be easily checked that already the periodic $L_2$-discrepancy of the empty set is only $3^{-d/2}$ (see the forthcoming Lemma~\ref{leinit}) which is much smaller than $d 2^{-d/2}$. We will see in the next section that the  periodic $L_2$-discrepancy suffers from the curse of dimensionality.
\end{remark}

In discrepancy theory many results are based on averaging arguments. In order to make such results relevant for applications one requires explicit constructions of point sets with ``averge'' discrepancy.
It is easily checked that the average squared periodic $L_2$-discrepancy $${\rm avg}_2^{{\rm per}}(N,d):=\left(\int_{[0,1]^d}\cdots \int_{[0,1]^d} (L_{2,N}^{{\rm per}}(\{\bsx_1,\ldots,\bsx_N\}))^2 \rd \bsx_1 \ldots \rd \bsx_N\right)^{1/2}$$ equals $$\frac{1}{\sqrt{N}} \left(\frac{1}{2^d}-\frac{1}{3^d}\right)^{1/2}.$$ Hence, Theorem~\ref{thm1} shows that the periodic $L_2$-discrepancy of Korobov's $p$-set is almost of average order.

\begin{remark}\rm
There is some relation to {\bf Open Problem 40} in \cite[p.~57]{NW10}. There the authors ask for a construction of an $N$-element point set in dimension $d$ in time polynomial in $N$ and $d$ for which the $L_p$-discrepancy is less than the average $L_p$-discrepancy taken over all $N$-element point sets in dimension $d$. Here we have---up to a linear factor $d$---the answer to this question for the periodic $L_2$-discrepancy.
\end{remark}

\begin{remark}\rm
Our result also gives some information about the usual $L_2$-discrepancy: The average $L_2$-discrepancy is  $${\rm avg}_2(N,d):=\left(\int_{[0,1]^d}\cdots \int_{[0,1]^d} (L_{2,N}(\{\bsx_1,\ldots,\bsx_N\}))^2 \rd \bsx_1 \ldots \rd \bsx_N\right)^{1/2}$$ and this is well-known to be  $$\frac{1}{\sqrt{N}} \left(\frac{1}{2^d}-\frac{1}{3^d}\right)^{1/2}.$$
Using Proposition~\ref{pr1} we obtain that the root mean square $L_2$-discrepancy of the shifted $p$-sets is at most  $ \frac{d}{2^{d/2}}  \frac{1}{\sqrt{N}}$ and hence almost of average order. This also implies that there must exist a shift $\bsdelta^* \in [0,1)$ for the $p$-sets $\cP\in \{\cP^{{\rm Kor}}_p,\cQ^{{\rm Kor}}_{p^2},\cR^{{\rm Kor}}_{p^2}\}$ such that $$L_{2,N}(\cP+\bsdelta^*)\le  \frac{d}{2^{d/2}}  \frac{1}{\sqrt{N}}.$$ Whether the $L_2$-discrepancy of the $p$-sets itself satisfies the bound $ \frac{d}{2^{d/2}}  \frac{1}{\sqrt{N}}$ is an interesting open problem.
\end{remark}

\section{The curse of dimensionality for the periodic $L_2$-discrepancy}\label{cod}

In this section we are interested in the inverse $N_2^{{\rm per}}(\varepsilon,d)$ of periodic $L_2$-discrepancy which is the minimal number $N \in \NN$ for which there exists an $N$-element point set $\cP$ in $[0,1)^d$ whose periodic $L_2$-discrepancy is less then $\varepsilon$ times the initial periodic $L_2$-discrepancy. Since lower bounds become stronger if proved also for the weighted version, we consider here also the inverse of the weighted periodic $L_2$-discrepancy $N_2^{{\rm per, w}}(\varepsilon,d)$ and compare the results to the results for the unweighted case and to the case where only positive weights are allowed.

\begin{lemma}\label{leinit}
The initial periodic $L_2$-discrepancy is 
 $$L_{2,0}^{{\rm per}}=\left(\int_{[0,1]^d} \int_{[0,1]^d} {\rm volume}(B(\bsx,\bsy))^2 \rd \bsx \rd \bsy\right)^{1/2}=\frac{1}{3^{d/2}}.$$ 
\end{lemma}

We omit the easy proof of this lemma.

Hence, for $\varepsilon \in (0,1)$ and $d \in \NN$ 
$$N_2^{{\rm per}}(\varepsilon,d)=\min\left\{ N \in \NN \ : \ \exists \ \cP \subseteq [0,1)^d,\, |\cP|=N \mbox{ and } L_{2,N}^{{\rm per}}(\cP) \le \frac{\varepsilon}{3^{d/2}}\right\}.$$
Analogously, we define $N_2^{{\rm per,w}}(\varepsilon,d)$ and $N_2^{{\rm per,w+}}(\varepsilon,d)$ if arbitrary or only positive weights are allowed, respectively.

In information based complexity one is particularly interested in the dependence of the inverse discrepancy on $\varepsilon^{-1}$ and $d$. In particular, a polynomial dependence on $d$ would be favourable (see, e.g., \cite{NW10}). This however is not achievable in the case of the periodic $L_2$-discrepancy. Indeed, the following theorem shows that the periodic $L_2$-discrepancy suffers from the curse of dimensionality. 

\begin{theorem}\label{thm:curse}
For $\varepsilon \in (0,1)$ and $d \in \NN$ we have 
$$N_2^{{\rm per}}(\varepsilon,d)   
\ge \frac{1}{1+\varepsilon^2} \left(\frac{3}{2}\right)^d
\quad \text{and} \quad 
N_2^{{\rm per,w+}}(\varepsilon,d)
\ge (1-\varepsilon^2) \, \left(\frac{3}{2}\right)^d.
$$
Moreover, for any $\varepsilon_0 \in (0,1)$ there exists $c>0$ such that
$$ 
 N_2^{{\rm per,w}}(\varepsilon,d) 
  \ge
 c \, 1.0628^d
$$
for all $d\in \NN$ and $\varepsilon\in (0,\varepsilon_0)$.
\end{theorem}

\begin{proof}
The lower bound for $N_2^{{\rm per,w}}(\varepsilon,d) $ follows directly from the corresponding lower bound for the inverse of the non-periodic $L_2$-discrepancy, see \cite{NW01} or (9.17) in \cite{NW10}, together with Proposition \ref{pr1}.
Similarly, a lower bound 
$$N_2^{{\rm per}}(\varepsilon,d) \ge  N_2^{{\rm per,w+}}(\varepsilon,d) 
\ge (1-\varepsilon^2) \, 1.125^d$$
follows from the corresponding lower bound for the inverse of the non-periodic $L_2$-discrepancy, see \cite{SW98} or (9.16) in \cite{NW10}, together with Proposition \ref{pr1}.

To prove the better lower bound for $N_2^{{\rm per}}(\varepsilon,d)$ we observe that
according to \cite[p.~389-390]{HOe} we have $$L_{2,N}^{{\rm per}}(\cP)^2=-\frac{1}{3^d}+\frac{1}{N^2} \sum_{n,m=1}^N \prod_{j=1}^d \left(\frac{1}{3}+B_2(|x_{n,j}-x_{m,j}|)\right),$$ where $B_2(x)$ is the second Bernoulli polynomial, $B_2(x)=x^2-x+\frac{1}{6}$ and $x_{n,j}$ is the $j^{{\rm th}}$ coordinate of the point $\bsx_n$. Note that $B_2(0)=1/6$ and for $x\in [0,1]$ we have $B_2(x) \ge B_2(1/2)=-1/12$. Now we have
\begin{eqnarray*}
L_{2,N}^{{\rm per}}(\cP)^2 \ge  -\frac{1}{3^d}+\frac{1}{N^2} \sum_{n,m=1\atop n=m}^N \frac{1}{2^d} + \frac{1}{N^2} \sum_{n,m=1\atop n\not=m}^N  \left(\frac{1}{3}-\frac{1}{12}\right)^d\ge -\frac{1}{3^d} + \frac{1}{N} \frac{1}{2^d}.
\end{eqnarray*}
Hence, $$L_{2,N}^{{\rm per}}(\cP) \le \frac{\varepsilon}{3^{d/2}}$$ implies $$-\frac{1}{3^d} + \frac{1}{N} \frac{1}{2^d} \le\frac{\varepsilon^2}{3^d}$$ and therefore $$N \ge \frac{1}{1+\varepsilon^2} \left(\frac{3}{2}\right)^d.$$

The slightly worse lower bound for $N_2^{{\rm per,w+}}(\varepsilon,d)$ follows similarly using
$$L_{2,N}^{{\rm per}}(\cP,\bsw)^2=\frac{1}{3^d} - \frac{2}{3^d} \sum_{n=1}^N w_n
+ \sum_{n,m=1}^N w_n w_m \prod_{j=1}^d \left(\frac{1}{3}+B_2(|x_{n,j}-x_{m,j}|)\right).$$ 
If the weights $w_n$ are nonnegative, estimating the double sum by the diagonal terms, we have
$$
L_{2,N}^{{\rm per}}(\cP,\bsw)^2 \ge  \frac{1}{3^d} - \frac{2}{3^d} \sum_{n=1}^N w_n  +  \frac{1}{2^d} \sum_{n=1}^N w_n^2 \ge \frac{1}{3^d} - \frac{2^d N}{3^{2d}}.
$$
Now it is easily seen that
$$L_{2,N}^{{\rm per}}(\cP,\bsw) \le \frac{\varepsilon}{3^{d/2}}$$ 
implies 
$$N \ge (1-\varepsilon^2) \left(\frac{3}{2}\right)^d.$$
\end{proof}

\begin{remark}\rm
Let $N$ be the smallest prime number larger or equal to $$\left\lceil \left(\frac{3}{2}\right)^d \frac{d^2}{\varepsilon^2}\right\rceil=:M.$$ Then it follows from Theorem~\ref{thm1} that for Korobov's $p$-set $\cP_p^{{\rm Kor}}$ with $p=N$ we have $$L_{2,N}^{{\rm per}}(\cP_p^{{\rm Kor}}) \le \frac{\varepsilon}{3^{d/2}}$$ and hence   
$$N_2^{{\rm per,w+}}(\varepsilon,d) \le N < 2 M = 2   \left\lceil \left(\frac{3}{2}\right)^d \frac{d^2}{\varepsilon^2}\right\rceil,$$ where we used Bertrand's postulate, which tells us that $M \le N < 2 M$.

This means, that the term $3/2$ in the lower bounds in Theorem~\ref{thm:curse} is the exact basis for the exponential dependence of the information complexity in $d$.
\end{remark}

\noindent\textbf{Authors' addresses:}\\

\noindent {\sc Josef Dick}

\noindent School of Mathematics and Statistics, The University of New South Wales, Sydney NSW 2052, Australia, 
email: josef.dick(AT)unsw.edu.au\\

\noindent {\sc Aicke Hinrichs}

\noindent Institut f\"ur Analysis, Johannes Kepler Universit\"at Linz, Altenbergerstra{\ss}e 69, 4040 Linz, Austria,
email: aicke.hinrichs(AT)jku.at\\

\noindent {\sc Friedrich Pillichshammer}

\noindent Institut f\"ur Finanzmathematik und Angewandte Zahlentheorie, Johannes Kepler Universit\"at Linz, Altenbergerstra{\ss}e 69, 4040 Linz, Austria,
email: friedrich.pillichshammer(AT)jku.at

\end{document}